\documentclass[12pt]{amsart}
\newtheorem{theorem}{Theorem}
\newtheorem{proposition}{Proposition}
\newtheorem{lemma}{Lemma}
\usepackage{geometry}
\title{A note on the Liouville function in short intervals}

\author{Kaisa Matom\"aki}
\address{Department of Mathematics and Statistics, University of Turku,
20014 Turku, Finland}
\email{ksmato@utu.fi}

\author{Maksym Radziwi\l\l}
\address{Department of Mathematics, Rutgers University \\ Hill Center for the Mathematical Sciences \\
110 Frelinghuysen Rd., Piscataway, NJ 08854-8019 }
\email{maksym.radziwill@gmail.com}

\begin{document}
\maketitle

%{\tt In the Appendix we could just give the proof of the result
%for general multiplicative functions $f$ and $h = X^{\varepsilon}$
%and skip all the statement of lemmas, instead just quoting directly
%which lemma in our paper we use \\ Perhaps mention in our main
%paper that with Tao now have result for Chowla on average?}.

\begin{abstract}
In this note we give a short and self-contained proof that, for any $\delta > 0$, $\sum_{x \leq n \leq x+x^\delta} \lambda(n) = o(x^\delta)$ for almost all $x \in [X, 2X]$. We also sketch a proof of a generalization of such a result to general real-valued multiplicative functions. Both results are special cases of results in our more involved and lengthy recent pre-print.
\end{abstract}

\section{Introduction}

In our recent pre-print~\cite{MainPaper} we have proved (among other things) the following
theorem.
\begin{theorem} \label{mainthm}
Let $f: \mathbb{N} \rightarrow [-1,1]$ be a multiplicative function, and let $h = h(X) \rightarrow \infty$, arbitrarily slowly with $X \rightarrow
\infty$. Then, for almost all $X \leq x \leq 2X$, 
$$
\frac{1}{h}\sum_{x \leq n \leq x + h} f(n) = \frac{1}{X} \sum_{X \leq n \leq 2X} f(n) + o(1)
$$ 
with $o(1)$ not depending on $f$.
\end{theorem}
In particular for the Liouville function this result implies that, 
for any $\psi(X) \to \infty$ with $X \to \infty$, we have
\begin{equation} \label{almostall}
\sum_{x \leq n \leq x + \psi(X)} \lambda(n) = o(\psi(X))
\end{equation}
for almost all $X \leq x \leq 2X$. Previously this was known unconditionally only when $\psi(X) \geq X^{1/6}$ (using zero-density theorems), and under the density hypothesis for $\psi(X) \geq X^\delta$ for any $\delta > 0$.

The proof of Theorem \ref{mainthm} is complicated for two reasons.
First of all, in order to achieve the result for a specific function
such as $\lambda(n)$ with $h$ growing arbitrarily
slowly we need to perform a messy decomposition of the Dirichlet polynomial
$$
\sum_{n \sim X} \frac{\lambda(n)}{n^{1 + it}}
$$
according to the size of $\sum_{P < p < Q} \lambda(p) p^{-1 -it}$ for suitable intervals
$[P,Q]$. Secondly, to obtain the result for arbitrary $f$, we need
to input some additional ideas dealing with large values of Dirichlet
polynomials. We realized recently that in the special case of
the Liouville function and intervals of length $X^{\delta}$ neither
is necessary.

Our goal in this short note is to give a short and self-contained
proof of the following special case of Theorem~\ref{mainthm}. 
\begin{theorem} \label{thm2}
Let $\delta > 0$ be given. Then, for almost all $X \leq x \leq 2X$, we have
$$
\sum_{x \leq n \leq x + X^{\delta}} \lambda(n) = o(X^{\delta}).
$$
\end{theorem}
We have not tried to optimize any of the bounds for the amount of cancellations
or for the size of the exceptional set. With a bit additional effort this can be
done (but we refer the reader to our paper \cite{MainPaper}). 

For the convenience of the reader we have also indicated in the appendix how to generalize
this result to arbitrary multiplicative $f$. This is more intricate and depends
on a number of lemmas which are proven in our paper \cite{MainPaper}. We will invoke
these lemmas freely throughout the proof of the following theorem.
\begin{theorem} \label{thm3}
Let $f: \mathbb{N} \rightarrow [-1,1]$ be a multiplicative function. Let $\delta > 0$.
Then, for almost all $X \leq x \leq 2X$, we have
$$
\frac{1}{X^\delta} \sum_{x \leq n \leq x + X^\delta} f(n) = \frac{1}{X} \sum_{X \leq n \leq 2X} f(n) + o(1)
$$
with $o(1)$ not depending on $f$. 
\end{theorem}

It is worthwhile to point out that essentially the only non-standard idea from~\cite{MainPaper} that is needed in the proof of Theorem \ref{thm2} is the use of Ramar\'e type identity (see~\eqref{eq:Ramare} below). 
%Theorem \ref{thm3}
%depends roughly on the ideas which enter with dealing with the set $\mathcal{U}$, 
%in the notation of \cite{MainPaper}. 
%If one wanted to push Theorem \ref{thm2} to hold true in shorter intervals, one would
%need to deal with the sets $\mathcal{T}_j$ as in our paper. If one also wanted
%this to hold for arbitrary $f$, then, in addition, it would be necessary to deal
%with the set $\mathcal{U}$, and this would recover Theorem \ref{mainthm}. 

%In order to push both results towards smaller
%intervals we would need to perform the inductive decomposition into sets 
%$\mathcal{T}_j$ as in our paper. If one just wished to obtain Theorem 1 in 
%smaller intervals it would be enough to deal with sets $\mathcal{T}_j$. 

\section{The main propositions}

Theorem \ref{thm2} follows immediately from the following proposition.
\begin{proposition} \label{proposition1}
Let $\delta > 0$ be given. Then, for any $\varepsilon > 0$,
$$
\int_{X}^{2X} \Big | \frac{1}{X^{\delta}} \sum_{x \leq n \leq x + X^{\delta}} \lambda(n) \Big |^2 dx \ll_{\varepsilon} \frac{X}{(\log X)^{1/3 - \varepsilon}}. 
$$
\end{proposition}
\begin{proof}[Deduction of Theorem \ref{thm2} from Proposition \ref{proposition1}]
By Chebyschev's inequality the number of exceptional $x \in [X, 2X]$ for which
$$
\Big | \frac{1}{X^{\delta}} \sum_{x \leq n \leq x + X^{\delta}} \lambda(n) \Big | \geq \frac{1}{(\log X)^{1/9}}. 
$$
is at most 
$$
(\log X)^{2/9} \int_{X}^{2X} \Big | \frac{1}{X^{\delta}} \sum_{x \leq n \leq x + X^{\delta}} \lambda(n) \Big |^2 dx 
\ll_{\varepsilon} \frac{X}{(\log X)^{1/9 - \varepsilon}} = o(X)
$$
as claimed. 
\end{proof}

In order to prove Theorem \ref{thm3} we will sketch the proof of the following proposition in the Appendix.
\begin{proposition}
\label{proposition2}
Let $f: \mathbb{N} \rightarrow [-1,1]$ be a multiplicative function.
Let $\delta > 0$ be given. Then
$$
\int_{X}^{2X} \Big | \frac{1}{X^{\delta}} \sum_{x \leq n \leq x + X^{\delta}} f(n) - \frac{1}{X} \sum_{X \leq n \leq 2X}
f(n) \Big |^2 dx \ll \frac{X}{(\log X)^{1/48}}.
$$
\end{proposition}

\section{Lemmas}
\label{se:lemmas}
In Lemma~\ref{lemma3} below we relate the integral in Proposition~\ref{proposition1} to a mean square of a Dirichlet polynomial. To deal with this, we use the following three standard lemmas.
\begin{lemma} \label{lemma1}
Let $A > 0$ be given. We have, uniformly in $|t| \leq (\log X)^{A}$,
$$
\sum_{n \sim X} \frac{\lambda(n)}{n^{1 + it}} \ll (\log X)^{-A}.
$$
\end{lemma}
\begin{proof}
By the prime number theorem for any $A > 0$, we have,
$$
\sum_{X \leq n \leq u} \frac{\lambda(n)}{n} \ll (\log X)^{-2A}
$$
for any $u \in [X, 2X]$. 
Therefore, integrating by parts we find
\begin{align*}
\sum_{n \sim X} \frac{\lambda(n)}{n^{1 + it}} & = \int_{X}^{2X} u^{-it} d \sum_{X \leq n \leq u} \frac{\lambda(n)}{n} \\
& \ll \frac{|t|}{X} \int_{X}^{2X} \Big | \sum_{X \leq n \leq u} \frac{\lambda(n)}{n} \Big | du + (\log X)^{-2A}\\
& \ll (\log X)^{A} \cdot (\log X)^{-2A} = (\log X)^{-A}
\end{align*}
which gives the claim. 
\end{proof}

\begin{lemma} \label{lemma2}
Let $A > 0$ be given and $X \geq 1$. Assume that $\exp((\log X)^{\theta}) \leq P \leq Q \leq X$ for some $\theta > 2/3$ and let
$$
\mathcal{P}(1 + it) = \sum_{P \leq p \leq Q} \frac{1}{p^{1 + it}}. 
$$
Then, for any $|t| \leq X$,
$$
|\mathcal{P}(1 + it)| \ll \frac{\log X}{1 + |t|} + (\log X)^{-A}. 
$$
\end{lemma}
\begin{proof}
In case $|t| \leq 10$, the claim follows immediately from the prime number theorem, so we can assume $|t| > 10$. We can also assume that fractional parts of $P$ and $Q$ are $1/2$ each. Perron's formula says that, for any $\kappa > 0$ and $y  > 0$, we have
$$
\frac{1}{2\pi i} \int_{\kappa - iT}^{\kappa + iT} y^s \cdot \frac{ds}{s} = \begin{cases}
1 & \text{ if } y > 1 \\
0 & \text{ if } y < 1 
\end{cases}
\quad + O \Big ( \frac{y^{\kappa}}{\max(1, T |\log y|)} \Big ).
$$
Therefore, letting $\kappa = 1 / \log X$, and $T = (|t|+1)/2 < |t|-1$, we have
\begin{align} \label{contour}
\mathcal{P}(1 + it)   = \frac{1}{2\pi i} \int_{\kappa - i T}^{\kappa + iT} & \log \zeta(s + 1 + it) \cdot \frac{Q^s - P^s}{s} \cdot ds
+ O\left(\frac{\log X}{|t|+1} + \frac{1}{P^{1/2}}\right).  
\end{align}
Using Vinogradov's zero-free region, we see that $\log \zeta(s + 1 + it)$ is well defined in the region
$$
\mathcal{R} : 1 \leq |\Im s + t| \leq 2X \ , \ \Re s \geq -\sigma_0 := - \frac{1}{(\log X)^{2/3} (\log\log X)}.
$$
In addition for $s \in \mathcal{R}$ we have $|\log \zeta(s + 1 + it)| \ll (\log X)^2$. Therefore shifting the 
contour in (\ref{contour}) to the edge of this region, we see that
\begin{align*}
\mathcal{P}(1 + it) & = \frac{1}{2\pi i} \int_{-T}^{T} \log \zeta  (1 - \sigma_0 + iu + it) \cdot \frac{Q^{-\sigma_0 + iu} - P^{-\sigma_0 + iu}}{-\sigma_0 + iu} du + O \Big ( \frac{\log X}{|t|+1} +\frac{1}{P^{1/2}}\Big ) \\
& \ll (\log X)^{-A} + \frac{\log X}{|t|+1}. 
\end{align*}
as claimed. 
\end{proof}

\begin{lemma} 
\label{le:mvt}
One has
$$
\int_{-T}^{T} \Big | \sum_{n \sim X} \frac{a_n}{n^{1 + it}} \Big |^2 \ll
(T + X) \sum_{n \sim X} \frac{|a_n|^2}{n^2}. 
$$
\end{lemma}
\begin{proof}
See \cite[Theorem 9.1]{IwKo04}.
\end{proof}

\section{Proof of Proposition 1}
We start with the following lemma which is in the spirit of previous work on primes in almost all intervals, see for instance~\cite[Lemma 9.3]{Harman06}.
\begin{lemma} \label{lemma3}
Let $\delta > 0$ be given. Then
\begin{align*}
&\frac{1}{X} \int_{X}^{2X} \Big | \frac{1}{X^{\delta}} \sum_{x \leq n \leq x + X^{\delta}}
\lambda(n) \Big |^2 dx \\
&\ll \int_0^{X^{1 - \delta}} \Big | \sum_{n \sim X} \frac{\lambda(n)}{n^{1 + it}} \Big |^2 dt
+ \max_{T > X^{1 - \delta}} \frac{X^{1 - \delta}}{T} \int_{T}^{2T} \Big | \sum_{n \sim X} \frac{\lambda(n)}{n^{1 + it}} \Big |^2 dt.
\end{align*}
\end{lemma}
\begin{proof}

Write $h := X^\delta$. By Perron's formula
\begin{align*}
\frac{1}{h} \sum_{x \leq n \leq x + h} \lambda(n) & = \frac{1}{h} \cdot 
\frac{1}{2\pi i} \int_{1 - i \infty}^{1 + i\infty} \Big ( \sum_{n \sim X} 
\frac{\lambda(n)}{n^{s}} \Big ) \cdot  \frac{(x + h)^s - x^s}{s} ds.
\end{align*}
Hence it is enough to bound
\[
\begin{split}
V := \frac{1}{h^2 X} \int_{X}^{2X} \left|\int_{1}^{1 + i\infty}  F(s) \frac{(x+h)^s-x^s}{s}ds\right|^2 dx,
\end{split}
\]
where $F(s) = \sum_{n \sim X} \lambda(n)n^{-s}$. We would like to add a smoothing, take out a factor $x^s$, expand the square, exchange the order of integration and integrate over $x$. However, the term $(x+h)^s$ prevents us from doing this and we overcome this problem in a similar way to \cite[Page 25]{SaVa77}. We write
\[
\begin{split}
&\frac{(x+h)^s - x^s}{s} = \frac{1}{2h} \left(\int_{h}^{3h} \frac{(x+w)^s - x^s}{s} dw - \int_{h}^{3h} \frac{(x+w)^s - (x+h)^s}{s} dw\right) \\
&= \frac{x}{2h} \int_{h/x}^{3h/x} x^s\frac{(1+u)^s - 1}{s} du - \frac{x+h}{2h} \int_{0}^{2h/(x+h)} (x+h)^s \frac{(1+u)^s - 1}{s} du.
\end{split}
\]
where we have substituted $w = x \cdot u$ in the first integral and $w = h+ (x+h)u$ in the second integral. Let us only study the first summand, the second one being handled completely similarly. Thus we assume that
\[
\begin{split}
V &\ll \frac{X}{h^4} \int_X^{2X}\left|\int_{h/x}^{3h/x} \int_{1}^{1+i\infty} F(s) x^s\frac{(1+u)^s - 1}{s} ds du \right|^2 dx \\
&\ll \frac{1}{h^3} \int_{h/(2X)}^{3h/X} \int_{X}^{2X}\left|\int_{1}^{1+i\infty} F(s) x^s\frac{(1+u)^s - 1}{s} ds \right|^2dx du \\
&\ll \frac{1}{h^2X} \int_{X}^{2X}\left|\int_{1}^{1+i\infty} F(s) x^s\frac{(1+u)^s - 1}{s} ds \right|^2dx
\end{split}
\]
for some $u \ll h/X$.

Let us introduce a smooth function $g(x)$ supported on $[1/2, 4]$ and
equal to $1$ on $[1,2]$. We obtain
%
%$[X/2, 4X]$ and $1$ on $[X, 2X]$. We obtain
\begin{align*}
V &\ll \frac{1}{h^2X} \int g\Big ( \frac{x}{X} \Big ) \left|\int_{1}^{1+i\infty} F(s) x^s\frac{(1+u)^s - 1}{s} ds \right|^2dx \\
&\leq \frac{1}{h^2 X} \int_{1}^{1+i\infty} \int_{1}^{1+i\infty} \left|F(s_1) F(s_2) \frac{(1+u)^{s_1}-1}{s_1} \frac{(1+u)^{s_2}-1}{s_2}\right| \left| \int g \Big ( \frac{x}{X} \Big ) x^{s_1+\overline{s_2}} dx \right| |ds_1 ds_2| \\
&\ll  \frac{1}{h^2X}\int_{1}^{1+i\infty} \int_{1}^{1+i\infty} |F(s_1) F(s_2)| \min\left\{\frac{h}{X}, \frac{1}{|t_1|}\right\} \min\left\{\frac{h}{X}, \frac{1}{|t_2|}\right\} \frac{X^3}{|t_1-t_2|^2+1} |ds_1 ds_2| \\
&\ll  \frac{X^2}{h^2} \int_{1}^{1+i\infty} \int_{1}^{1+i\infty} \frac{|F(s_1)|^2 \min\{(h/X)^2, |t_1|^{-2}\}+ |F(s_2)|^2\min\{(h/X)^2, |t_2|^{-2}\}}{|t_1-t_2|^2+1} |ds_1 ds_2| \\
&\ll \int_{1}^{1+iX/h} |F(s)|^2 |ds| + \frac{X^2}{h^2} \int_{1+iX/h}^{1+i\infty} \frac{|F(s)|^2}{|t|^2} |ds|.
\end{align*}
The second summand is
\[
\ll \frac{X^2}{h^2} \int_{1+iX/(2h)}^{1+i\infty} \frac{1}{T^3} \int_{1+iT}^{1+i2T} |F(s)|^2 |ds| dT \ll \frac{X^2}{h^2} \cdot \frac{1}{X/h} \max_{T \geq X/(2h)} \frac{1}{T} \int_{1+iT}^{1+i2T} |F(s)|^2 |ds| 
\]
and the claim follows.
\end{proof}
Proposition \ref{proposition1} will follow from combining Lemma \ref{lemma3} with
the following lemma.
\begin{lemma} \label{lemma4}
Let $\delta > 0$ be given. 
Then, 
\begin{align*}
\int_{0}^{T} \Big | & \sum_{n \sim X} \frac{\lambda(n)}{n^{1 + it}} \Big |^2 dt \ll_{\varepsilon} \frac{1}{(\log X)^{1/3 - \varepsilon}} 
\cdot \Big ( \frac{T}{X} + 1 \Big ) + \frac{T}{X^{1 - \delta/2}}.
%\frac{1}{(\log X)^{A}} \cdot \Big ( \frac{T}{X^{1 - \delta/2}} + 1 \Big ). 
\end{align*}
\end{lemma}
\begin{proof}
Since the mean value theorem (Lemma~\ref{le:mvt}) gives the bound $O(\frac{T}{X}+1)$, we can assume $T \leq X$. Furthermore, by Lemma~\ref{lemma1}, the part of the integral with $t \leq T_0 := (\log X)^{10}$ contributes $O((\log X)^{-10})$.

Let us now concentrate to the integral over $[T_0, T]$ with $T \leq X$.
Let $P = \exp((\log X)^{2/3 + \varepsilon})$ and $Q = X^{\delta/3}$. We use the decomposition
\begin{equation}
\label{eq:Ramare}
\sum_{n \sim X} \frac{\lambda(n)}{n^{1 + it}} = \sum_{P \leq p \leq Q} \frac{\lambda(p)}{p^{1+it}} \sum_{\substack{m \sim X / p}} \frac{\lambda(m)}{(\# \{p \in [P, Q] \colon p \mid m\}+1)m^{1 + it}} + \sum_{\substack{n \sim X \\ p \mid n \implies p \not \in [P, Q]}}\frac{\lambda(n)}{n^{1 + it}},
\end{equation}
which is a variant of Ramar\'e's identity \cite[Section 17.3]{Opera}. Writing $a_m = \lambda(m)/(\# \{p \in [P, Q] \colon p \mid m\}+1)$, we obtain
\begin{equation}
\begin{split}
\label{eq:intest}
\int_{T_0}^{T} \Big | \sum_{n \sim X} \frac{\lambda(n)}{n^{1 + it}} \Big |^2 dt & \ll \int_{T_0}^{T}
\Big | \sum_{P \leq p \leq Q} \frac{1}{p^{1+it}} \sum_{\substack{m \sim X / p}} \frac{a_m}{m^{1 + it}}\Big|^2 dt + \\
& + \int_{T_0}^{T} \Big | \sum_{\substack{n \sim X \\ p \mid n \implies p \not \in [P, Q]}}\frac{\lambda(n)}{n^{1 + it}}\Big|^2 dt.
\end{split}
\end{equation}
We estimate the second term by completing the integral to $|t| \leq T$ and by applying the mean-value theorem (Lemma \ref{le:mvt}).
This shows that the second term is bounded by
$$
\ll (T + X) \frac{1}{X^2} \sum_{\substack{n \sim X \\ p \mid n \implies p \not \in [P, Q]}} 1  \ll \Big ( \frac{T}{X} + 1 \Big ) \cdot \frac{\log P}{\log Q}
\ll \frac{1}{(\log X)^{1/3 - \varepsilon}} \cdot \Big ( \frac{T}{X} + 1 \Big )
$$
by the fundamental lemma of the sieve. To deal with the first term in~\eqref{eq:intest}, we would like to dispose of the condition $mp \sim x$, so that we can use lemmas in Section~\ref{se:lemmas} to Dirichlet polynomials over $p$ and $m$ separately. To do this, we let $H = (\log X)^5$ and split the summations in the appearing Dirichlet polynomial into short ranges, getting
\begin{equation}
\label{eq:dyadsplit}
\sum_{P \leq p \leq Q} \frac{1}{p^{s}} \sum_{\substack{m \sim X / p}} \frac{a_m}{m^{s}} = \sum_{\lfloor H \log P \rfloor \leq j \leq H \log Q } \ \sum_{\substack{e^{j/H} \leq p < e^{(j+1)/H} \\ P \leq p \leq Q}} \frac{1}{p^s} \ \sum_{\substack{X e^{-(j+1)/H} \leq m \leq 2X e^{-j/H} \\ X \leq m p \leq 2X}} \frac{a_m}{m^s}.
\end{equation}
Now we can remove the condition $X \leq mp \leq 2X$ over-counting at most by the integers $mp$ in the ranges $[X e^{-1/H}, X]$ and $[2X, 2X e^{1/H}]$. Therefore we can, for some bounded $d_m$, rewrite~\eqref{eq:dyadsplit} as
\begin{align*}
\sum_{\lfloor H \log P \rfloor \leq j \leq H \log Q} & Q_{j, H}(s) F_{j, H}(s)+  \sum_{\substack{X e^{-1/H} \leq m \leq X}}\frac{d_m}{m^s} + 
\sum_{\substack{2X \leq m \leq 2X e^{1/H}}} \frac{d_m}{m^s} 
\end{align*}
where
$$
Q_{j,H}(s) := \sum_{e^{j/H} \leq p \leq e^{(j+1)/H}} \frac{1}{p^s} \ \text{ and } \ F_{j,H}(s) := \sum_{X e^{-(j+1)/H} \leq m \leq 2X e^{-j/H}}
\frac{a_m}{m^s}. 
$$
Using this decomposition, applying Cauchy-Schwarz and then taking the maximal term in the resulting sum, we get
\begin{align*}
\int_{T_0}^{T} & \Big | \sum_{P \leq p \leq Q} \frac{1}{p^{1 + it}} \sum_{m \sim X / p} \frac{a_m}{m^{1 + it}} \Big |^2 dt 
 \ll (H \log (Q/P))^2 \int_{T_0}^{T} \Big | Q_{j,H}(1 + it) F_{j, H}(1 + it) \Big |^2 dt + \\ & + \int_{T_0}^{T}
\Big | \sum_{X e^{-1/H} \leq m \leq X} \frac{d_m}{m^{1 + it}} \Big |^2 dt + \int_{T_0}^{T}
\Big | \sum_{2X \leq m \leq 2X e^{1/H}} \frac{d_m}{m^{1 + it}} \Big |^2 dt .
\end{align*}
for some $j \in [\lfloor H \log P \rfloor , H \log Q]$ depending at most on $X$ and $T$. 
We compute the last two integrals by completing the integral to $|t| \leq T$, and applying the mean value theorem (Lemma 
\ref{le:mvt}). This way we see that they
are bounded by
$$
\ll (T + X) \frac{1}{X^2} \cdot (X e^{1/H} - X) \ll \Big ( \frac{T}{X} + 1 \Big ) \frac{1}{H} = \frac{1}{(\log X)^5} \Big ( \frac{T}{X} + 1 \Big ). 
$$
Finally, since $X^{\delta/3} = Q \geq e^{j/H} \geq P/e > \exp((\log X)^{2/3 + \varepsilon/2})$, 
using Lemma~\ref{lemma2} we have, for $T_0 \leq t \leq X$, 
$$
|Q_{j,H}(1 + it)| \ll (\log X)^{-9}.
$$
Therefore, by the mean value theorem (Lemma~\ref{le:mvt}),
\begin{align*}
\int_{T_0}^{T} |Q_{j,H}(1 + it)F_{j,H}(1 + it)|^2 dt & \ll (\log X)^{-18} \int_{T_0}^{T} |F_{j,H}(1 + it)|^2 dt \\
& \ll (\log X)^{-18} \cdot (T + X e^{-j/H}) \frac{1}{X e^{-j/H}} \\ & \ll (\log X)^{-18} \cdot \Big ( \frac{Q T}{X} + 1 \Big ) \ll \frac{T}{X^{1 - \delta/3}} + \frac{1}{(\log X)^{18}}
\end{align*}
since $e^{j/H} \leq Q = X^{\delta/3}$. Combining everything together we get the following bound
$$
\int_{0}^{T} \Big | \sum_{n \sim X} \frac{\lambda(n)}{n^{1 + it}} \Big |^2 dt
\ll \frac{1}{(\log X)^{1/3 - \varepsilon}} \cdot \Big ( \frac{T}{X} + 1 \Big ) 
+ (\log X)^{12} \cdot \Big ( \frac{T}{X^{1 - \delta/3}} + \frac{1}{(\log X)^{18}} \Big )
$$
which implies the required result.
\end{proof}
We are now ready to prove Proposition~\ref{proposition1}.
\begin{proof}[Proof of Proposition~\ref{proposition1}]
Using Lemma \ref{lemma4} we get 
$$
\int_{0}^{X^{1 - \delta}} \Big | \sum_{n \sim X} \frac{\lambda(n)}{n^{1 + it}} \Big |^2 dt
\ll \frac{1}{(\log X)^{1/3-\varepsilon}}
$$ 
and similarly
$$
\max_{T > X^{1 - \delta}} \frac{X^{1 - \delta}}{T} \int_{T}^{2T} \Big | \sum_{n \sim X} \frac{\lambda(n)}{n^{1 + it}}
\Big |^2 dt \ll \frac{1}{(\log X)^{1/3-\varepsilon}}.
$$
We conclude therefore using Lemma \ref{lemma3}, that, 
$$
\frac{1}{X} \int_{X}^{2X} \Big | \frac{1}{X^\delta} \sum_{x \leq n \leq x + X^{\delta}} \lambda(n) \Big |^2 dx \ll \frac{1}{(\log X)^{1/3-\varepsilon}}
$$
as claimed.
\end{proof}
\section{Appendix: Proof of Proposition \ref{proposition2}}
The proof of Proposition \ref{proposition2} is more involved and involves
more tools. We will therefore freely make appeal to \cite{MainPaper} whenever
necessary. First,~\cite[Lemma 14]{MainPaper} (a variant of Lemma~\ref{lemma3} here), implies that in order to establish Proposition~\ref{proposition2}, we need to bound
$$
\int_{(\log X)^{1/15}}^{T} \Big | \sum_{n \sim X} \frac{f(n)}{n^{1 + it}} \Big |^2 dt.
$$
and perform a minor cosmetic operation. %to relate the average over $X \leq n \leq X (\log X)^{-1/5}$ to the
%average over the full interval $X \leq n \leq 2X$. 
The main ingredient in the proof of Proposition~\ref{proposition2} is thus the following lemma.
\begin{lemma}
\label{le:mainest}
We have, 
$$
\int_{(\log X)^{1/15}}^{T} \Big | \sum_{n \sim X} \frac{f(n)}{n^{1 + it}} \Big |^2 dt \ll
\frac{1}{(\log X)^{1/48}} \cdot \Big ( \frac{T}{X} + 1\Big ) + \frac{T X^{o(1)}}{X}.
$$
\end{lemma}
\begin{proof}
In view of the trivial bound $O(T/X + 1)$ from the mean value theorem (Lemma 
\ref{le:mvt}) we can assume that $T \leq X$. 

Let 
\[
H = (\log X)^{1/48}, \quad P = \exp((\log X)^{1 - 1/48}), \quad Q = \exp(\log X / \log\log X),
\]
and let
$$
Q_{j,H}(s) := \sum_{e^{j/H} \leq p \leq e^{(j+1)/H}} \frac{f(p)}{p^s} \quad \text{and}\quad F_{j,H}(s) := 
\sum_{X e^{-(j+1)/H} \leq m \leq 2X e^{-j/H}} \frac{f(m)}{m^s}.
$$
Then using~\cite[Lemma 12]{MainPaper} (which is a slightly more involved version of some of the arguments in proof of Lemma~\ref{lemma4}) we find the following bound,
\begin{align*}
& \int_{(\log X)^{1/15}}^{T} \Big | \sum_{n \sim X} \frac{f(n)}{n^{1 + it}} \Big |^2 dt \ll 
\\ & \ll (\log X)^{2 + 1/24} \int_{(\log X)^{1/15}}^{T} |Q_{j,H}(1 + it)F_{j,H}(1 + it)|^2 dt 
+ \frac{1}{(\log X)^{1/48}} \cdot \Big ( \frac{T}{X} + 1 \Big )
\end{align*}
for some $\lfloor H \log P \rfloor \leq j \leq H \log Q$ depending at most on $T$ and $X$.   

Let us define
\[
\begin{split}
\mathcal{T}_S  &= \{t \in [(\log X)^{1/15}, T] : |Q_{j,H}(1 + it)| \leq (\log X)^{-100}\} \\
\text{and} \quad \mathcal{T}_L  &= \{t \in [(\log X)^{1/15}, T] : |Q_{j,H}(1 + it)| > (\log X)^{-100}\}.
\end{split}
\]
On $\mathcal{T}_S$ we have by definition and the mean value theorem (Lemma~\ref{le:mvt})
\begin{align*}
\int_{\mathcal{T}_S} \Big | Q_{j,H}(1 + it) F_{j,H}(1 + it)|^2 dt & \ll (\log X)^{-200} \int_{0}^{T} |F_{j,H}(1 + it)|^2 dt \\
& \ll (\log X)^{-200} \cdot (T + X e^{-j/H}) \cdot \frac{1}{X e^{-j/H}} \\
& \ll (\log X)^{-200} \cdot \Big ( \frac{T X^{o(1)}}{X} + 1 \Big )
\end{align*}
since $e^{j/H} \leq Q = X^{o(1)}$, which is a sufficient saving in the logarithm since we need to beat $(\log X)^{2 + 1/24}$ by at least $(\log X)^{1/48}$. 

Let us now turn to $\mathcal{T}_L$. We can find a well-spaced subset $\mathcal{T} \subseteq \mathcal{T_L}$ such that
\[
\int_{\mathcal{T}_L} \Big | Q_{j,H}(1 + it) F_{j,H}(1 + it)|^2 dt \ll \sum_{t \in \mathcal{T}} \Big | Q_{j,H}(1 + it) F_{j,H}(1 + it)|^2 dt
\]
Using \cite[Lemma 8]{MainPaper}, we see that
\begin{align*}
|\mathcal{T}| & \ll \exp \Big ( 2 \frac{\log (\log X)^{100}}{j/H} \log T + 2 \log (\log X)^{100} + 2 \frac{\log T}{j/H} \log\log T \Big ) \\
& \ll \exp \Big ( \frac{(\log X)^{1 + o(1)}}{\log P} \Big ) \ll \exp((\log X)^{1/48 + o(1)}). 
\end{align*}
In addition, using \cite[Lemma 3]{MainPaper} (a consequence of Hal\'asz's theorem), we find that
$$
\sup_{(\log X)^{1/15} \leq |t| \leq T} |F_{j,H}(1 + it)| \ll (\log X)^{-1/16} \cdot \frac{\log Q}{\log P} \ll (\log X)^{-1/24}.
$$
Therefore using~\cite[Lemma 11]{MainPaper} (a large value result for Dirichlet polynomials over primes) this time, we get
\begin{align*}
& \sum_{t \in \mathcal{T}} |Q_{j,H}(1 + it) F_{j,H}(1 + it)|^2  \ll (\log X)^{-1/12}
\sum_{t \in \mathcal{T}}
|Q_{j,H}(1 + it)|^2 \\
& \ll (\log X)^{-1/12} \cdot \Big (  e^{j/H} + |\mathcal{T}| e^{j/H} \exp( - (\log X)^{1/5}) \Big )
\sum_{e^{j/H} < p < e^{(j+1)/H}} \frac{1}{p^2 \log p} \\
& \ll (\log X)^{-1/12} \cdot \frac{e^{(j+1)/H}-e^{j/H}}{e^{j/H}(\log e^{j/H})^2} \ll \frac{1}{H (\log X)^{1/12} (\log P)^{2}} \ll (\log X)^{-2-1/16}.
\end{align*}
Therefore combining everything together we get
\begin{align*}
\int_{(\log X)^{1/15}}^{T} \Big | \sum_{n \sim X} \frac{f(n)}{n^{1 + it}} \Big |^2 dt & \ll (\log X)^{-100} \Big ( \frac{T X^{o(1)}}{X} + 1 \Big )
+ (\log X)^{-1/48} \Big ( \frac{T}{X} + 1 \Big ),
\end{align*}
and the claim follows.
\end{proof}

We are finally ready to prove the Proposition.
\begin{proof}[Proof of Proposition 1]
Now, by~\cite[Lemma 14]{MainPaper} and Lemma~\ref{le:mainest}, we get, for $X^\delta = h_1 \leq h_2 = X / (\log X)^{1/5}$, 
\[
\frac{1}{X} \int_{X}^{2X} \left|\frac{1}{h_1} \sum_{\substack{x \leq m \leq x+h_1}} f(m) - \frac{1}{h_2} \sum_{\substack{x \leq m \leq x+h_2}} f(m) \right|^2 dx \ll (\log X)^{-1/48}.
\]
The claim follows since by~\cite[Lemma 4]{MainPaper}
\[
\frac{1}{h_2} \sum_{x \leq n \leq x + h_2} f(n) = \frac{1}{x} \sum_{X \leq n \leq 2X} f(n) + O((\log X)^{-1/20}).
\]
\end{proof}

%\bibliographystyle{plain}
%\bibliography{../biblio}
\end{document}